\documentclass[11pt]{amsart}

\usepackage{amssymb,amsmath,amsthm,amscd,mathrsfs,graphicx}
\usepackage[cmtip,all]{xy}
\usepackage{amscd}
\usepackage{pb-diagram,pb-xy} 

\begin{document}

\theoremstyle{plain} \numberwithin{equation}{section}
\newtheorem*{unlabeledtheorem}{Theorem}
\newtheorem{theorem}{Theorem}[section]
\newtheorem{corollary}[theorem]{Corollary}
\newtheorem{fact}[theorem]{Fact}
\newtheorem{conjecture}{Conjecture}
\newtheorem{lemma}[theorem]{Lemma}
\newtheorem{proposition}[theorem]{Proposition}
\newtheorem{convention}[theorem]{}
\newtheorem*{unlabeledproposition}{Proposition}
\newtheorem{theoremalpha}{Theorem}
\renewcommand{\thetheoremalpha}{\Alph{theoremalpha}}
\newtheorem{hypothesis}{Hypothesis}
\renewcommand{\thehypothesis}{\Alph{hypothesis}}
\newtheorem{propositionalpha}[theoremalpha]{Proposition}
\newtheorem{corollaryalpha}[theoremalpha]{Corollary}

\theoremstyle{definition}
\newtheorem{definition}[theorem]{Definition}
\newtheorem{finalremark}[theorem]{Final Remark}
\newtheorem{remark}[theorem]{Remark}
\newtheorem*{unlabeledremark}{Remark}
\newtheorem{example}[theorem]{Example}
\newtheorem{question}{Question} 
\newtheorem*{warning}{Warning}

\newcommand{\Spec}{\text{Spec}}
\newcommand{\Sets}{\text{Sets}}
\newcommand{\Sch}{\text{Sch}}
\newcommand{\Groups}{\text{Grp}}

\newcommand{\ord}{\text{ord}}
\newcommand{\degree}{\text{deg}}

\newcommand{\calA}{{ \mathcal A}}
\newcommand{\calB}{{ \mathcal B}}
\newcommand{\calC}{{ \mathcal C}}
\newcommand{\calD}{{ \mathcal D}}
\newcommand{\calE}{{ \mathcal E}}
\newcommand{\calF}{{ \mathcal F}}
\newcommand{\calG}{{ \mathcal G}}
\newcommand{\calH}{{ \mathcal H}}
\newcommand{\calI}{{ \mathcal I}}
\newcommand{\calJ}{{ \mathcal J}}
\newcommand{\calK}{{ \mathcal K}}
\newcommand{\calL}{{ \mathcal L}}
\newcommand{\calM}{{ \mathcal M}}
\newcommand{\calN}{{ \mathcal N}}
\newcommand{\calO}{{ \mathcal O}}
\newcommand{\calP}{ {\mathcal P} }
\newcommand{\calQ}{{ \mathcal Q}}
\newcommand{\calR}{{ \mathcal R}}
\newcommand{\calS}{{ \mathcal S}}
\newcommand{\calT}{{ \mathcal T}}
\newcommand{\calU}{{ \mathcal U}}
\newcommand{\calV}{{ \mathcal V}}
\newcommand{\calW}{{ \mathcal W}}
\newcommand{\calX}{{\mathcal X}}
\newcommand{\calY}{{\mathcal Y}}
\newcommand{\calZ}{ {\mathcal Z}}

\newcommand{\Pic}{\operatorname{P}}
\newcommand{\PicO}{\operatorname{P}^{\underline{0}}}
\newcommand{\PicT}{\operatorname{P}^{\tau}}
\newcommand{\PicZero}{\operatorname{P}^{0}}

\newcommand{\Eid}{\operatorname{E}}
\newcommand{\EidO}{\operatorname{E}^{\operatorname{o}}}
\newcommand{\EidT}{\operatorname{E}^{\tau}}

\newcommand{\Quot}{\operatorname{Q}}
\newcommand{\QuotO}{\operatorname{Q}^{\operatorname{o}}}
\newcommand{\QuotT}{\operatorname{Q}^{\tau}}

\newcommand{\Neron}{\operatorname{N}}

\newcommand{\Sh}{\operatorname{Sh}}
\newcommand{\Sheaf}{\operatorname{Sheaf}}

\makeatletter
\@namedef{subjclassname@2010}{\textup{2010} Mathematics Subject 
Classification} \makeatother

\keywords{N\'{e}ron model, stable sheaf, compactified Jacobian, relative Picard functor}
\subjclass[2010]{Primary: 14H40; Secondary: 14H10, 14D22,  14L15.}
%\subjclass[2010]{Primary: 14H10; Secondary: 14H40, 14D22,  14L15.}
\thanks{The author was partially supported by NSF grant DMS-0502170.}

\title[Two ways]{Two ways to degenerate the Jacobian are the same}
\author[Kass]{Jesse Leo Kass}
\address{University of Michigan, Department of Mathematics,  530 Church St, Ann Arbor, MI 48103, USA}
\email{kass@math.harvard.edu}

\date{\today}

\bibliographystyle{alpha}

\begin{abstract}
	We provide sufficient conditions for the line bundle locus in a family of compact
	moduli spaces of pure sheaves to be isomorphic to the N\'{e}ron model. The result
	applies to moduli spaces constructed by Eduardo Esteves and Carlos Simpson, 
	extending results of Busonero, Caporaso, Melo, Oda, Seshadri, and Viviani.
\end{abstract}

\maketitle

\section{Introduction}
\subsection{Background}
This paper relates two different approaches to extending families of Jacobian varieties.  Recall that if 
$X_0$ is a smooth projective curve of genus $g$, then the associated Jacobian variety  is a
 $g$-dimensional smooth projective variety $J_{0}$ that
 can be described  in two different ways: as the universal abelian variety that contains $X_0$ (the Albanese variety), and as the moduli space of  degree $0$ line bundles on $X_0$ (the Picard variety).  If $X_{U} \to U$ is a family of smooth, 
 projective curves, then the Jacobians of the fibers fit together to form a family $J_{U} \to U$.  In this paper, 
$U$ will be an open subset of a smooth curve $B$ (or, more generally, a Dedekind scheme), 
and we will be interested in extending $J_{U}$ to a family over $B$. Corresponding to the two 
different ways of describing the Jacobian (Albanese vs. Picard) are two different approaches to extending the family
$J_{U} \to U$.

Viewing the Jacobian as the Albanese variety, it is natural to try to extend $J_{U} \to U$ by extending 
it to a family of group varieties over $B$.  N\'{e}ron showed that this can be done in a canonical way in 
\cite{neron64}.  He worked with an arbitrary  family of abelian varieties
$A_{U} \to U$ and proved that there is a unique $B$-smooth group scheme $\Neron := \Neron(A_{U}) \to B$
extending $A_{U} \to U$ which is universal with respect to a natural mapping property.  This scheme is called
the N\'{e}ron model.  Arithmetic geometry has seen the use of  the N\'{e}ron  model in a number of important results
 (e.g.,  \cite{mazur}, \cite{mazur77}, \cite{wiles}, \cite{gross}).  The N\'{e}ron model of a Jacobian
variety plays a particularly prominent role, and an alternative description of this scheme
in terms of the relative Picard functor was given by Raynaud (\cite{raynaud70}).  We primarily
work with Raynaud's description, which is recalled in Section~\ref{Sec: RaynaudReview}.

An alternative approach, suggested by viewing the Jacobian as the Picard variety,
is to extend $J_{U} \to U$ as a family of moduli spaces of sheaves.
This approach was first proposed by Mayer and Mumford in \cite{mayer}.   Typically,
one first extends $X_{U} \to U$ to a family of curves $X \to B$
and then extends $J_{U}$ to a family $\bar{J} \to B$ with 
the property that the fiber over a point $b \in B$ is a moduli space of sheaves on $X_b$ parameterizing certain line bundles, 
together with their degenerations.  In this paper, we show that the line bundle locus $J$  in $\bar{J}$
is canonically isomorphic to the N\'{e}ron model  for some schemes $\bar{J}$ constructed in the literature.

To state this more precisely, we need to specify which schemes $\bar{J}$ we consider.  
The problem of constructing such a family of moduli spaces has been studied by many mathematicians who have 
constructed different compactifications (e.g.,  \cite{ishida},  \cite{dsouza}, \cite{oda}, \cite{altman80}, \cite{caporaso94}, \cite{simpson94}, \cite{pand}, \cite{jarvis}, \cite{esteves}).  Many of the difficulties to performing such a construction arise from the fact that, when $X_b$ is reducible, the degree $0$ line bundles on a fiber $X_b$ do not form a bounded family.

For simplicity, assume the residue field $k(b)$ is algebraically closed and $X_b$ is reduced with components labeled 
$X_1, \dots, X_n$.  Given a line bundle $\calM$ of degree $0$ on $X_b$, 
the sequence $(\deg(\calM|_{X_1}), \dots, \deg(\calM|_{X_n}))$ is called the multi-degree of $\calM$. This sequence 
must sum to $0$, but may otherwise be arbitrary, which implies unboundedness.  A bounded family can be 
obtained by fixing the multi-degree, and typically  the scheme $\bar{J}$ is defined so that it parameterizes (possibly coarsely) line bundles (and their degenerations) that satisfy a numerical condition on the multi-degree.  This paper focuses on constructions given by   Simpson ({\cite{simpson94}) and by Esteves (\cite{esteves}), which we now describe in more detail.  

For the Simpson moduli space, the numerical condition imposed on line bundles is slope semi-stabilty with respect to an auxiliary 
 ample line bundle.  This condition arises from the method of construction: the moduli space is constructed using Geometric Invariant Theory, 
and slope stability corresponds to GIT stability.  In general, the Simpson moduli space is a coarse moduli 
space in the sense that non-isomorphic sheaves may correspond to the same point of the space, but it contains 
an open subscheme (the stable locus) that is a fine moduli space, and we will work exclusively with this locus.  
Families of moduli spaces of semi-stable sheaves have been constructed for arbitrary families of projective schemes, but we will only be concerned with 
the case of families of curves.

The moduli spaces of Esteves parameterize sheaves that are $\sigma$-quasi-stable.  Like slope stability, $\sigma$-quasi-stabillity
is a numerical conditions on the multi-degree, but it is defined in terms of an auxiliary vector bundle $\calE$ and a section
$\sigma$, rather than an ample line bundle.  The moduli spaces are constructed for families over an arbitrary locally noetherian
base, but strong conditions are required of the fibers: they must be geometrically reduced.  The space is constructed as a closed subspace of a (non-noetherian, non-separated) algebraic space that was constructed in \cite{altman80}.  For nodal curves, the authors of \cite{viviani10} describe the relation between the Esteves moduli spaces  and  the Simpson moduli spaces.  However, here we treat these moduli spaces separately.

For a discussion of the relation between these moduli spaces and other moduli spaces constructed in the literature, the reader is directed to \cite{alexeev} and \cite[Sect.~2]{kass:local2}.
The reader familiar with the work of Caporaso is warned that there is one potential source of confusion.  In \cite{caporaso94}, the compactified Jacobian associated to a stable curve $X$ parameterizes pairs $(Y,L)$ consisting of a line bundle $L$ on a nodal curve $Y$ stably equivalent to $X$ that satisfies certain conditions.  The line bundle locus $J$ that we study corresponds to the locus parameterizing pairs $(Y,L)$ with $X=Y$.

\subsection{Main Result}The main result of this paper relates the line bundle locus in a proper family of moduli spaces of sheaves to the N\'{e}ron model of the Jacobian:

\begin{theoremalpha}  \label{Thm: StatementPolished}
Let $f \colon X \to B$ be a family of geometrically reduced curves over a Dedekind scheme with $X$ regular.  Define the following schemes:
\begin{enumerate}
	\item $\mathcal{J} \subset \bar{\mathcal{J}}$ is the line bundle locus in one of the following families of moduli spaces:
	\begin{itemize}
		\item the Esteves compactified Jacobian or
		\item the Simpson compactified Jacobian associated to a polarization such  that slope semi-stability coincides with slope stability;
	\end{itemize}
	\item $\operatorname{N}$ is the N\'{e}ron model of the Jacobian of the generic fiber $X_{\eta}$.
\end{enumerate}
Then $$\mathcal{J} = \operatorname{N}.$$
\end{theoremalpha}
Theorem~\ref{Thm: StatementPolished} is the combination of Corollaries \ref{Cor: EstevesJacDedekind} and \ref{Cor: SimpIsNeron}, which themselves are consequences of Theorem~\ref{Thm: Main Theorem}.  Theorem~\ref{Thm: Main Theorem} is quite general, and we expect that it applies to many other fine moduli spaces of sheaves  (but NOT coarse ones).  In particular, Theorem~\ref{Thm: Main Theorem} applies to families of curves with possibly non-reduced fibers, though then general results asserting the existence of a suitable moduli space are unknown (but see Sect.~\ref{Subsec: Genus1} for some simple examples).

The arithmetically-inclinced reader should note that Theorem~\ref{Thm: StatementPolished} and the results later in this paper do not place any hypotheses on the base Dedekind scheme $B$.  In particular, we do not assume that  the residue fields are perfect.  The author was initially surprised by this as there is a body of work (e.g. \cite{qing}, \cite{raynaud70}) showing that various pathologies can arise when $k(b)$ fails to be perfect.

Theorem~\ref{Thm: StatementPolished} has interesting consequences for both the N\'{e}ron model and the compactified Jacobian.   One consequence of the theorem is that N\'{e}ron models of Jacobians can often be constructed over high dimensional bases. The N\'{e}ron model of an abelian variety is only defined over a (regular) $1$-dimensional base $B$, but no such dimensional hypotheses are need to apply the existence results from  \cite{simpson94} and \cite{esteves}.  At the end of Section~\ref{Subsec: Genus1},
we examine a family $J \to \mathbb{P}^2$ over the plane with the property that $\mathbb{P}^2$ is covered by lines $C$ such that the restriction $J_{C}$ is the N\'{e}ron model of its generic fiber.  Surprisingly, while the N\'{e}ron models fit into a $2$-dimensional family, their group structure does not.  

Theorem~\ref{Thm: StatementPolished} also has interesting consequences for the moduli spaces of Esteves and Simpson.  Indeed, if $f \colon X \to B$ is a family of curves satisfying the hypotheses of the theorem, then both the Esteves Jacobians $J_{\calE}^{\sigma}$ and the Simpson Jacobians $J_{\calL}^{0}$ (for $\calL$ as in the hypothesis) are independent of the particular polarizations, and every such Simpson Jacobian is isomorphic to every Esteves Jacobian.  This is not immediate from the definitions.  At the end of Section~\ref{Subsec: Esteves}, we discuss this fact in greater detail and pose a related question.

\subsection{Past Results}  Certain cases of Theorem~\ref{Thm: StatementPolished} were already known.  In his thesis (\cite{busonero}), Simone Busonero established Theorem~\ref{Thm: StatementPolished} for certain Esteves Jacobians.  The thesis is unpublished, but a different proof using similar techniques that extends the result  to the Simpson moduli spaces can be found in the pre-print \cite[Thm.~3.1]{viviani10} by Melo and  Viviani.  They prove Theorem~\ref{Thm: StatementPolished} when the fibers of $f$ are nodal and $X$ is regular.  We do not discuss the Caporaso universal compactified Jacobian here, but the relation between that scheme and the N\'{e}ron model was described by Caporaso in \cite{caporaso08}, \cite{caporaso}, and \cite{caporaso10} (esp. \cite[Thm.~2.9]{caporaso10}). Earlier still, Oda and Seshadri related their $\phi$-semi-stable compactified Jacobians, also not discussed here, to N\'{e}ron models (\cite[Cor.~14.4]{oda}).  In each of those papers, an important step in the proof is  a combinatorial argument establishing that, for example, the natural  map from the set of $\sigma$-quasi-stable multi-degrees to the degree class group is a bijection.

The proof given here does not use any combinatorics, and the idea can be described succinctly.  Consider the special case where
$B := \Spec( \mathbb{C}[[t]] )$, which is a strict henselian discrete valuation ring with algebraically closed residue field.  There is a natural map $J \to \Neron$ to the N\'{e}ron model coming from the universal property of $\Neron$, and an application of Zariski's Main Theorem shows that this morphism is an open immersion. Thus, the only issue is set-theoretic surjectivity.  Because $B$ is henselian, every point on the special fiber of $\Neron$ is the specialization of a section, so surjectivity is equivalent to the surjectivity of the map $J(\mathbb{C}[[t]]) \to J( \operatorname{Frac} \mathbb{C}[[t]] )$ that sends a section to its restriction to the generic fiber.  A given point
$p \in J( \operatorname{Frac} \mathbb{C}[[t]] )$ may be extended to a section $\sigma \in \bar{J}(\mathbb{C}[[t]])$ of $\bar{J}$ by the valuative criteria.  As $\bar{J}$ is a fine moduli space, $\sigma$ corresponds to a family of rank $1$, torsion-free sheaves, which in fact must be a family of line bundles because $X$ is factorial.  We may conclude that $\sigma \in J( \mathbb{C}[[t]] )$, yielding the result.

\subsection{Questions} It would be interesting to know when 
a Simpson Jacobian $J_{\calL}^{0}$ satisfying the hypotheses of Theorem~\ref{Thm: StatementPolished}
exists; that is, given a family $f \colon X \to B$, does there exist an ample
line bundle $\calL$ such that every $\calL$-slope semi-stable sheaf of degree $0$ is stable?  This question seems most interesting 
when the geometric fibers of $f$ are non-reduced.  We briefly survey the literature on this topic at the end of Section~\ref{Subsec: Simpson}.

More generally, given a family $f \colon X \to B$, it would be desirable to have a description of the maximal
subfunctors of the degree $0$ relative Picard functor $\PicZero$ representable by a separated $B$-scheme.  
We discuss this question in Section~\ref{Subsec: Genus1}, where we analyze the simple case of genus $1$ curves.

\subsection{Organization} We end this introduction with a few technical remarks about the paper.  The moduli spaces of sheaves
that we consider are moduli spaces of pure sheaves.  On a curve, a coherent sheaf is pure if and only if it is 
Cohen--Macaulay.  This condition is also equivalent to the condition of being torsion-free
in the sense of elementary algebra when the curve is integral, and 
the term ``torsion-free" is sometimes used in place of ``pure."  

The term  ``family of curves" only to refers to families with geometrically irreducible generic
fibers.  This is done to avoid notational complications concerning multi-degrees.   Families of curves 
are required to be proper but not projective.  A family of curves over a Dedekind scheme
can fail to be projective (e.g. \cite[8.10]{altman80}), but projectivity is automatic if the local rings of 
the total space are factorial, which is the main case of interest.  (See Prop.~\ref{Prop: Projective}.)

We prove the main results for families over a base  scheme $S$ that is the spectrum of a strict henselian discrete valuation ring rather than a
more general Dedekind scheme.  Doing so lets us make section-wise arguments because a smooth family of a henselian base
admits many sections.  Furthermore, this is not a real restriction: results over a general Dedekind base can be deduced 
by passing to the strict henselization. 

The body of the paper is organized as follows.   In Section~\ref{Sec: RaynaudReview}, we review 
Raynaud's construction of the maximal separated quotient.  We then relate this scheme is to 
a general moduli space of line bundles satisfying some axioms in Section~\ref{Sec: MainThm}.  Finally, we describe some schemes
that satisfy these axioms in the  final section, Section~\ref{Sec: Apps}.

\section*{Conventions}
\begin{convention}
	The symbol $X_{T}$ denotes the fiber product $X \times_{S} T$.  
\end{convention}

\begin{convention}
	The letter $S$ denotes the spectrum of a strict henselian discrete valuation ring with special point $0$ and generic point $\eta$.
\end{convention}

\begin{convention}
	A curve over a field $k$ is a proper $k$-scheme $f_0 \colon X_0 \to \Spec(k)$ that is geometrically connected and of pure  dimension $1$.
\end{convention}

\begin{convention}
	If $B$ is a scheme, then a family of curves over $B$ is a proper, flat morphism
$f \colon X \to B$ whose fibers are curves and whose geometric generic fibers are irreducible.
\end{convention}

\begin{convention}
	If $f \colon Y \to B$ is a finitely presented morphism, then we write $Y^{\text{sm}} \subset Y$ for the smooth 
	locus of $f$.
\end{convention}

\begin{convention}
	A coherent module $I_0$ on a noetherian scheme $X_0$ is
	rank $1$ if the localization of $I_0$
	at $x$ is isomorphic to  $\calO_{X_0, x}$ for every generic point $x$.
\end{convention}

\begin{convention}
	A coherent module $I_0$ on a noetherian scheme $X_0$ is pure if the dimension of $\operatorname{Supp}(I_0)$
	equals the dimension of $\operatorname{Supp}(J_0)$ for every non-zero subsheaf $J_0$ of $I_0$.
\end{convention}

\begin{convention}
	If $X_0 \to \Spec(k)$ is proper, then the degree of a coherent $\calO_{X_0}$-module
	$\calF$ is defined by $\degree(\calF) := \chi(\calF) - \chi(\calO_{X})$.
\end{convention}

\section{Raynaud's Maximal Separated Quotient} \label{Sec: RaynaudReview}
We begin by reviewing Raynaud's construction of the N\'{e}ron model of a Jacobian and, more generally, the 
maximal separated quotient of  the relative Picard functor  (\cite{raynaud70}).  Much of this material is also 
treated in \cite[Chap.~9]{bosch90}.  

Let $S$ be a strict henselian discrete valuation ring with generic point $\eta$
and special point $0$.  Given a family of curves $f \colon X \to S$, the \textbf{relative Picard functor} $\Pic$ of $f$ is defined 
to be the \'{e}tale sheaf $\Pic \colon \text{$S$--$\Sch$} \to \Groups$ associated to functor 
\begin{equation}  \label{Eqn: PicPresheaf}
	T \mapsto \operatorname{Pic}(X_{T}).
\end{equation}
Here $\operatorname{Pic}(X_{T})$ is the set of isomorphism classes of line bundles on $X_{T}$.
Raynaud  actually defines $\Pic$ to be the associated fppf sheaf but then observes
that this is the same as the associated \'{e}tale sheaf (\cite[1.2]{raynaud70}; see also \cite[Rmk.~9.2.11]{kleiman05}). 
The fibers of $\Pic$ are representable by group schemes locally of finite type, and
$\Pic$ itself is representable by an algebraic space if and only if $f$ is cohomologically flat
(\cite[Thm.~5.2]{raynaud70}).  Regardless of the representability properties, $\Pic$
is locally finitely presented and formally $S$-smooth.

Inside of $\Pic$, we may consider the subfunctor $\Eid \colon \text{$S$--$\Sch$} \to \Groups$ that is defined to be
 the scheme-theoretic closure of the identity 
section.  This is the fppf subsheaf of $\Pic$ generated by the elements
$g \in \Pic(T)$, where $T \to S$ is flat and $g_{\eta} \in \Pic(T_{\eta})$
is the identity element.  When $\Pic$ is a scheme, this coincides with the 
usual notion of closure.  The representability properties of $\Eid$ are similar to those
of $\Pic$: the fibers of $\Eid$ are group schemes locally of finite type  and $\Eid$ is 
representable by an algebraic space precisely when $f$ is cohomologically flat (\cite[Prop.~5.2]{raynaud70}).
When representable, $\Eid \to S$ is an \'{e}tale $S$-group space;  in general, the generic
fiber of $\Eid$ is the trivial group scheme, and the special fiber is a group scheme of dimension
equal to $h^{0}(\calO_{X_0}) - h^{0}(\calO_{X_{\eta}})$.  

When $\Eid$ is not the trivial $S$-group scheme, $\Pic$ does not satisfy the valuative
criteria of separatedness.  We can, however, form the \textbf{maximal separated quotient}
  $\Quot \colon \text{$S$--$\Sch$} \to \Groups$ of $\Pic$.  By definition, this is the fppf quotient sheaf 
$\Quot := \Pic/\Eid$.   The maximal separated quotient  $\Quot$ is always representable by a
scheme that is $S$-smooth, separated, and locally 
of finite type (\cite[Thm.~4.1.1, Prop.~8.0.1]{raynaud70}).   Rather than working directly with $\Quot$, we shall primarily work
with a slightly smaller subfunctor $\QuotT \colon \text{$S$--$\Sch$} \to \Groups$, which we now define.

Suppose generally that $B$ is a scheme and $G \colon \text{$B$--$\Sch$} \to \Groups$ is
a $B$-group functor whose fibers are representable by group schemes  locally of finite type.  
For every point $b \in B$, we may form the identity component $G^{\operatorname{o}}_{b} \subset G_{b}$ 
and the component group $G_{b}/G^{\operatorname{o}}_{b}$.  The subgroup
functor $G^{\tau} \subset G$ is defined to the subfunctor whose $T$-valued points
are elements $g \in G(T)$ with the property that, for every $t \in T$ mapping to $b \in B$,
the element $g_{t} \in G_{b}(k(t))$ maps to a torsion element of $G_{b}/G^{\operatorname{o}}_{b}(k(t))$.
If we instead require that $g_{t}$ maps to the identity element, then 
we obtain the subgroup functor $G^{\operatorname{o}} \subset G$.
Let us examine these constructions when $B$ equals $S$ and $G$ equals  $\Pic$ or  $\Quot$.

The functors $\Pic^{\operatorname{o}}$ and $\PicT$ coincide, and this common 
functor is the \'{e}tale sheaf associated to the assignment sending $T$ to
the set of isomorphism classes of line bundles on $X_{T}$ that fiber-wise have multi-degree $0$.
From this description, it is easy to see that $\Pic^{\operatorname{o}} = \PicT \subset \Pic$
is an open subfunctor.  Another open subfunctor of $\Pic$ is the subfunctor
parameterizing line bundles on $X_{T}$ with fiber-wise degree $0$, which
we denote by $\PicZero$.  It is typographically difficult to distinguish between
$\PicZero$ and $\Pic^{\operatorname{o}}$, but we will not make use of $\Pic^{\operatorname{o}}$
in this paper, so this should not cause confusion.

The functors $\QuotO$ and $\QuotT$ are different in general.  They 
are, however, both open subfunctors of $\Quot$ 
(\cite[Thm.~1.1({\romannumeral 1}.{\romannumeral 1}), Cor.~1.7]{groth95}).
In particular, they are both representable by smooth and separated $S$-group
schemes that are locally of finite type.  In fact,
both schemes are of finite type over $S$ as their fibers are easily
seen to have a finite number of connected components.  The condition
that $\QuotT \subset \Quot$ is a closed subscheme is 
important, but slightly subtle.  A characterization of this condition
is given by Prop.~8.1.2({\romannumeral 3}) of \cite{raynaud70}; one sufficient
(but not necessary) condition for $\QuotT \subset \Quot$ to be closed
is that the local rings of $X$ are factorial.

The factoriality condition is also
almost sufficient to ensure that $\QuotT$ is the N\'{e}ron model of 
its generic fiber.  Suppose that the generic fiber of $f$ is 
smooth, so the generic fiber of $\QuotT \to S$ is an abelian
variety, and thus it makes sense to speak of the N\'{e}ron model $\Neron := \Neron(\QuotT_{\eta})$.  
By the universal property, there is a unique morphism
$\QuotT \to \Neron$ that is the identity on the generic fiber.  
Theorem~8.1.4  of \cite{raynaud70} states that if the local
rings of $X$ are factorial, then $\QuotT \to \Neron$ is an isomorphism
in the following cases: when $k(0)$ is perfect and when
a certain invariant $\delta$ is coprime to the residual characteristic.

The proof uses the characterization of the N\'{e}ron
model in terms of the weak N\'{e}ron mapping 
property.  Recall that a $S$-scheme $Y \to S$ is said
to be a weak N\'{e}ron model of its generic 
fiber if the natural map $Y(S) \to Y(\eta)$ is bijective.  If 
$G \to S$ is a finite type $S$-group scheme whose generic fiber
is an abelian variety, then $G$ is the N\'{e}ron model of its
generic fiber if and only if it satisfies the weak N\'{e}ron
mapping property (\cite[Sec.~7.1, Thm.~1]{bosch90}).  

\section{The Main Theorem} \label{Sec: MainThm}
Here we derive the main results for families over a strict henselian 
discrete valuation ring $S$.  Specifically, we provide sufficient condition
for the maximal separated quotient $\QuotT$ of the Picard functor to be the N\'{e}ron 
model, and we relate $\QuotT$ to a fine moduli  space of line bundles that satisfies certain axioms. 
These moduli spaces are, by definition, subfunctors of a (large) functor that we now define.

\begin{definition} \label{Def: Sh}
	If $T$ is a $S$-scheme, then we define 
	$\Sheaf(X_{T})$ to be the set of isomorphism classes of
	$\calO_{T}$-flat, finitely presented $\calO_{X_T}$-modules
	$\calI$ on $X_{T}$ that 
	are fiber-wise pure, rank $1$, and
	of degree $0$.

	The functor $\Sh = \Sh_{X/S} \colon \text{$S$--$\Sch$} \to \Sets$
	is defined to be the \'{e}tale sheaf associated to the functor
	\begin{equation} \label{Eqn: ShRule}
		T \mapsto \Sheaf(X_{T}).
	\end{equation}
\end{definition}
There is a tautological 
transformation $\PicZero \to \Sh$ that realizes $\PicZero$ as a subfunctor
of $\Sh$.  In fact:

\begin{lemma} \label{Lemma: Subfunctor}
	The subfunctor  $\PicZero \subset \Sh$ is open.
\end{lemma}
\begin{proof}
	Given a $S$-scheme $T$ and a morphism $g \colon T \to \Sh$, we must show
	that $T \times_{\Sh} \PicZero$ is representable by a scheme and
	that $T \times_{\Sh} \PicZero \to T$ is an open immersion.  Thus, let $g$ be given.
		
	By definition, there exists an \'{e}tale surjection $p \colon T' \to T$
	and a sheaf $\calI' \in \Sheaf(X_{T'})$ that represents 
	$g \circ p \colon T' \to \Sh$.  
	Consider the subset $U' \subset T'$ of points $t \in T'$ with the
	property that  the restriction of $\calI$ to the 
	fiber $X_{t}$ is a line bundle.  This locus is
	open by  \cite[Lemma~5.12(a)]{altman80},
	and one may easily show that $U'$ represents
	$T' \times_{\Sh} \PicZero$.  A descent argument
	establishes the analogous property for the
	image $U$ of $U'$ under $T' \to T$.  
	This completes the proof.
\end{proof}

A remark about topologies: we work with the \'{e}tale
sheaf associated to Equation~\ref{Eqn: ShRule},
but one could instead work with the associated
fppf sheaf.  When $f$ is projective; it is a theorem
of Altman--Kleiman that the subfunctor of
$\Sh$ parameterizing simple sheaves
can be represented by a quasi-separated, locally 
finitely presented $S$-algebraic space, and hence is a fppf sheaf 
(\cite[Thm.~7.4]{altman80}).  
We do not know if $\Sh$ is an fppf sheaf in general.  Here
$\Sh$ is just used as a tool for keeping track
of schemes, and certainly any representable
subfunctor of $\Sh$  is a fppf sheaf.

One reason for working with the \'{e}tale topology instead of the
fppf topology is that it makes the following fact easy to prove.
\begin{fact} \label{Lemma: SectionsAreFamilies}
	The natural map $\Sheaf(X) \to \Sh(S)$ is surjective.
\end{fact}
\begin{proof}
	Let $g \in \Sh(S)$ be given.  By definition, there is an \'{e}tale morphism
	$S' \to S$ and an element $\calI' \in \Sheaf(X_{S'})$ that maps
	to $g_{S'} \in \Sh(S')$.  But $S$ is strict henselian, so $S' \to S$ 
	may be taken to be an isomorphism $S \to S$ (\cite[Prop.~18.8.1(c)]{ega44}), in which case the
	result is obvious.  
\end{proof}

The following two facts about separably closed fields are standard, but they 
will be used so frequently that it is convenient to record them.

\begin{fact} \label{Fact: SepPoints}
	If $k(0)$ is a separably closed field and $f_0 \colon Y_0 \to \Spec(k(0))$ 
	is smooth of relative dimension $n$, then the closed points of $Y_0$ with residue field $k(0)$ 
	are dense.
\end{fact}
\begin{proof}
	This is \cite[Cor.~13]{bosch90}.  The scheme $Y_0$ can be covered by affine 
	opens $U_0$ that admit an \'{e}tale morphism $p \colon U_0 \to \mathbb{A}^{n}_{k(0)}$.
	Certainly, the closed points with residue field $k(0)$ are dense in the 
	image of $p$.  If $v_0 \in \mathbb{A}^{n}_{k(0)}$ is one such point, then $p^{-1}(v_0)$
	is a finite, \'{e}tale $k(0)$-algebra, hence a disjoint union of closed points
	defined over $k(0)$.  Density follows.
\end{proof}

Fact~\ref{Fact: SepPoints} is typically used together with the following 
fact to assert that a smooth morphism has many sections.

\begin{fact}\label{Fact: SectionsExist}
	Let $Y \to S$ be a smooth morphism over strict henselian discrete valuation
	ring.  Then $Y(S) \to Y(k(0))$ is surjective.  
\end{fact}
\begin{proof}
	This is \cite[Cor.~17.17.3]{ega44} (or \cite[Prop.~14]{bosch90}).  If $U$ and $X'$ are as in the statement of \cite[Cor.~17.17.3]{ega44}, then we 
	must have $U = S$ and $X' \to U$ may be taken to be an isomorphism 
	(again, by \cite[Prop.~18.8.1(c)]{ega44}).
\end{proof}

We now prove the main results of the paper.
\begin{proposition} \label{Prop: OpenImmersion} 
	Let $f \colon X \to S$ be a family of curves and $J \subset \PicZero$ a subfunctor such that
the generic fibers $J_{\eta} =\PicZero_{\eta}$ coincide.  Assume $J$ is represented by a 
smooth, finitely presented  $S$-scheme.  

If $J$ is $S$-separated, then $J \to \Quot$ is an open immersion.
Furthermore, the image is contained in $\QuotT$ provided $\QuotT \subset \Quot$ is closed (e.g., 
the local rings of $X$ are factorial).
\end{proposition}
\begin{proof}
	This is an application of Zariski's Main Theorem.  We begin by showing that
	the induced map $J \to \Quot$ is injective on closed points.  It is enough to
	verify this after extending base $S$ so that $k(0)$ is algebraically closed.	
  	Thus, we will temporarily assume $k := k(0)$ is algebraically closed and work with 
	$k$-valued points instead of closed points.  Given $q \in \Quot(k)$, there is nothing 
 	show when the fiber over $q$ is empty.  If non-empty, pick $p \in J(k)$ mapping to $q$.
	We may invoke Fact~\ref{Fact: SectionsExist} to assert that there
	exists a section $\sigma \in J(S)$ with $\sigma(0)=p$.  
		
	The fiber of $q$ under $\Pic \to \Quot$ is the set of elements of the form
	$p + e$ with $e \in \Eid(k)$ or, equivalently, the elements of $(\sigma + \Eid)(k)$ (\cite[Cor.~4.1.2]{raynaud70}).
	Restricting to $J$, we see that the fiber of $q$ under $J \to \Quot$ is the set of $k$-valued points
	of $(\sigma + \Eid) \cap J$.
	But $(\sigma + \Eid) \cap J$ is  the scheme-theoretic closure of $\sigma$ in $J$ (by \cite[2.8.5]{ega42}), which is 
	just the image of $\sigma$ by separatedness.  In particular, the pre-image
	of $q$ under $J \to \Quot$ must be the singleton set $\{ p \}$.  This proves 
	that the map is injective on closed points.  We now return to the case
	where $S$ is a henselian discrete valuation ring (so $k(0)$ is no longer
	 assumed to be algebraically closed).
	
	It follows that the set-theoretic fibers of $J \to \Quot$ are finite sets.  Indeed, if 
	$Z \subset J$ is the locus of points $x \in J$ with the property that  $x$ lies in a 
	positive dimensional fiber, then $Z$ is closed by Chevalley's Theorem 
	(\cite[13.1.3]{ega42}).  Furthermore, $Z$ is contained in the special fiber $J_0$
	and contains no closed points.  This is only possible if $Z$ is the empty scheme.
	In other words, the set-theoretic fibers of $J \to \Quot$ are $0$-dimensional,
	hence finite (by  \cite[14.1.9]{ega41}).
	
	It follows immediately from Zariski's Main Theorem (\cite[4.4.9]{ega31}) that $J \to \Quot$ is an open immersion.  
	This proves the first part of the theorem.  To complete the proof, observe that flatness implies that
	 the generic fiber of $J_{\eta}$ is dense in $J$ (\cite[2.8.5]{ega42}). In particular, $J$
	is contained in the closure of $J_{\eta}$ in $\Quot$.  The generic fiber of $J$ coincides with the
	generic fiber of $\QuotT$, so the closure of this common scheme is contained in $\QuotT$
	when $\QuotT \subset \Quot$ is closed.  This completes the proof.
\end{proof}
\begin{remark}
	In Proposition~\ref{Prop: OpenImmersion}, we do \emph{not} assume that $J \subset \PicZero$ is an open subfunctor, but this condition holds in most cases of 
	interest.  When open, $J$ is automatically formally smooth and locally of finite presentation.  Thus, the key hypothesis in 
	the proposition is that $J$ is represented by a $S$-separated scheme.  A similar remark holds for
	Theorem~\ref{Thm: Main Theorem}; there the key hypotheses are that  $\bar{J}$ satisfies the valuative criteria of properness
	and that $J$ is representable.   Indeed, we do not even need to assume that $\bar{J}$ is representable.
\end{remark}

Under stronger assumptions, we can actually show that the natural map $J \to \QuotT$ is an isomorphism.
The essential point is to prove that $J$ satisfies the weak N\'{e}ron mapping property.  When $J$ 
can be embedded in a $S$-proper moduli space $\bar{J}$, this property holds provided that the local rings of $X$ are
factorial.  The content of this claim is that a line bundle on the generic fiber can specialize only to a line bundle on the
special fiber.  By localizing, the claim is equivalent to the following lemma, which is based on a proof from 
\cite{altman79} (p.~27 after ``Step XII").

\begin{lemma} \label{Lemma: LineBundleSpecialize}
	Suppose that $(R, \pi)$ is a discrete valuation ring and $R \to \calO$ a local, flat algebra extension
with $\calO$ noetherian.  Let $M$ be a $R$-flat,  finite $\calO$-module with the property 
that $M[\pi^{-1}]$ is free of rank $1$ and $M/\pi M$ is a rank $1$, pure module.  If $\calO$ is
factorial, then $M$ is free of rank $1$. 
\end{lemma}
\begin{proof}
	We can certainly assume $\calO$ is not the zero ring.  
To ease notation, we write $\bar{M} := M / \pi M$ and $\bar{\calO} := \calO/ \pi \calO$.
It is enough to prove that $M$ is isomorphic to a height  $1$ ideal.
Indeed, such an ideal is principal by the factoriality assumption.

We argue by first showing that $M$ is isomorphic to an ideal of $\calO$.  Let 
$\bar{\mathfrak{p}}_1, \dots, \bar{\mathfrak{p}}_n$ be the minimal primes of $\bar{\calO}$
and $\mathfrak{p}_1, \dots, \mathfrak{p}_n$ the corresponding primes of $\calO$.
We have assumed that the stalk $\bar{M} \otimes k(\bar{\mathfrak{p}}_i)$ is 1-dimensional.  This stalk 
coincides with the stalk $M \otimes k(\mathfrak{p}_i)$, so we can conclude that
the localization $M_{\mathfrak{p}_i}$ is free of rank $1$ for $i = 1, \dots, n$.

We can also conclude that the same holds for the localizations of the dual module
$M^{\vee} := \operatorname{Hom}(M, \calO)$.  An application of the Prime 
Avoidance Lemma shows that there exists an element $\phi \in M^{\vee}$ that 
maps to a generator of $M^{\vee}_{\mathfrak{p}_i}$ for all $i$.  We will show that
$\phi \colon M \to \calO$ realizes $M$ as a $R$-flat family of ideals
(i.e., $\phi$ is injective with $R$-flat cokernel).

 	It is enough to 
show that the reduction $\bar{\phi} \colon \bar{M} \to \bar{\calO}$ 
is injective.  This map factors as 
\begin{displaymath}
	\bar{M} \to \oplus \bar{M}_{\bar{\mathfrak{p}}_{i}} \to \oplus \bar{\calO}_{\bar{\mathfrak{p}}_i}.
\end{displaymath}
The kernel of the left-most map is a submodule whose support does not contain
any generic point of a component of $\Spec(\calO)$, and thus must be zero
by pureness.  Furthermore, the right-most map is an isomorphism by construction.

	Now consider the ideal $I[\pi^{-1}]$ given by the image of $\phi[\pi^{-1}] \colon M[\pi^{-1}] \to \calO[\pi^{-1}]$.  
This is a principal ideal, hence is either the unit ideal or an ideal of height at most $1$ (Hauptidealsatz!).
By flatness, the same is true of the image $I$ of $\phi$.  In fact, $I$ cannot be a height zero ideal: 
the only such prime is the zero ideal, which does not satisfy the hypotheses.  Thus $I$ is either
the unit ideal or a height $1$ ideal.  In either case, $I$ must be principal, and the proof is complete.
\end{proof}

The factorial condition is important in what follows, so we record it as a hypothesis.
\begin{hypothesis} \label{Hyp: Factorial}
	We say a family of curves $f \colon X \to B$ over a Dedekind scheme satisfies Hypothesis~\ref{Hyp: Factorial} if the generic fiber $X_{\eta}$ is smooth and the local rings of $X_{S}$ are factorial for every strict henselization $S \to B$.
\end{hypothesis}
Observe Hypothesis~\ref{Hyp: Factorial} is satisfied when $X$ is regular and $X_{\eta}$ is smooth.  We now prove the main theorem of this paper.
\begin{theorem} \label{Thm: Main Theorem}  
	Let $f \colon X \to S$ be a family of curves and $\bar{J}$ a subfunctor of $\Sh$ such that
the generic fibers $\bar{J}_{\eta} =\Sh_{\eta}$ coincide.  Assume the line bundle locus $J \subset \bar{J}$
is represented by a smooth and finitely presented $S$-scheme.

	If $\bar{J}$ satisfies the valuative criteria of properness and $f$ satisfies Hypothesis~\ref{Hyp: Factorial},
	then $\QuotT$ is the N\'{e}ron model and $$J \subset \QuotT = \Neron$$ is an open
	subscheme that contains all the $k(0)$-valued points of $\QuotT$.  Furthermore, $$J = \QuotT = \Neron$$ 
	provided one of the following conditions hold:
	\begin{enumerate}
		\item $k(0)$ is algebraically closed;
		\item $J$ is stabilized by the identity component $\QuotO$.
	\end{enumerate}
\end{theorem}
\begin{proof}
	By Proposition~\ref{Prop: OpenImmersion}, the natural map $J \to \Quot$ is 
	an open immersion with image contained in $\QuotT$.  Using this fact,
	we can prove that  $\QuotT$ is the N\'{e}ron model of its generic fiber.  Indeed, it is enough to 
	prove that $\QuotT$ satisfies the weak N\'{e}ron mapping property.  The open subscheme 
	$J \subset \QuotT$, in fact,  already satisfies this property.  Let $\sigma_{\eta} \in \QuotT(\eta) = J(\eta)$ be given.
  By properness, we can extend $\sigma_{\eta}$ to a section
$\sigma \in \bar{J}(S)$, and this element can be represented by 
a family $\calI$ of  pure, rank $1$ sheaves (by Fact~\ref{Lemma: SectionsAreFamilies}).
But every such family is a family of line bundles (Lemma~\ref{Lemma: LineBundleSpecialize}), and hence 
$\sigma$ lies in $J(S) \subset \bar{J}(S)$. In other words,  $J$ satisfies the weak N\'{e}ron mapping property.

	The weak N\'{e}ron mapping property of $J$ also implies that the image of
	$J$ contains all the $k(0)$-valued points of $\QuotT$.  Indeed, 
	every $k(0)$-valued point of $\QuotT$ is the specialization of a section
	by  Fact~\ref{Fact: SectionsExist}.  If $k(0)$ is algebraically closed,
	then we have shown that $J$ contains every $k(0)$-valued point
	of $\QuotT$, hence every closed point.  Thus, $J=\QuotT$,
	and there is nothing more to show.

	Let us now turn our attention to the case where $k(0)$ is only separably closed,
	but $J$ is stabilized by $\QuotO$.  Our goal is to show $J=\QuotT$, and to show this, 
	we pass to the special fiber $J_0 \to \QuotT_0$ and argue with points.
	Let $x$ be a $\bar{k}(0)$-valued point of $\QuotT$, where $\bar{k}(0)$ 
	is the algebraic closure of the residue field.  By density (Fact~\ref{Fact: SepPoints}),
	there exists a $\bar{k}(0)$-valued point $y$ in the image of $J_{0} \to \QuotT_0$ that lies in the
	same connected component as $x$.  We have $x = y + (x-y)$, which expresses $x$ as the sum 
	of a point of $\QuotO_0$ and a point of $J_{0}$.  The point $x$ must lie
	in $J_{0}$ by assumption.  This shows that the image of $J$ contains all of $\QuotT$, completing the proof.
\end{proof}

\begin{remark}
	The hypothesis that $J$ is stabilized by the identity component $\QuotO$ is perhaps unexpected,
	but it is often satisfied in practice.  The moduli space $\bar{J}$ is typically constructed by imposing numerical
	conditions on the multi-degree of a sheaf, and the multi-degree is invariant under the action of $\QuotO$
	(because the action is given by tensoring with a multi-degree $0$ line bundle).
\end{remark}

In the next section, we will show that certain moduli spaces
constructed in the literature satisfy the hypotheses 
of Theorem~\ref{Thm: Main Theorem}.  There are, however,
families of curves $f \colon X \to S$ with factorial local rings $\calO_{X,x}$ such that 
there does not exist a $S$-scheme $\bar{J}$ satisfying the conditions
 of the theorem.  Indeed, the family $f \colon X \to S$ in \cite[Ex.~9.2.3]{raynaud70} 
is a family of genus $1$ curves such that the local rings of $X$ are factorial
(even regular), but the natural map $\QuotT \to \Neron$ is not an isomorphism.
In particular, no $\bar{J}$ satisfying the hypotheses of Theorem~\ref{Thm: Main Theorem}
can exist.

\section{Applications} \label{Sec: Apps}
Here we apply Theorem~\ref{Thm: Main Theorem} to some families of moduli spaces from the literature
and then deduce consequences.  The two moduli spaces that we are interested in are the 
Esteves moduli space of quasi-stable sheaves (Sect.~\ref{Subsec: Esteves}) and the Simpson moduli space
of slope stable sheaves (Sect.~\ref{Subsec: Simpson}). In Section~\ref{Subsec: Genus1}, we discuss
the special case of families of genus $1$ curves, where suitable moduli spaces can be constructed explicitly.  

The moduli spaces we study are associated to a relatively projective family of curves.  We are primarily
interested in families over a Dedekind scheme with locally factorial total space, in which case projectivity
is automatic.  This fact is a consequence  of the generalized Chevalley Conjecture when the Dedekind 
scheme is defined over a field, but we do not know a reference in general.  For completeness, we prove:

\begin{proposition} \label{Prop: Projective}
	Let $f \colon X \to B$ be a family of curves over a Dedekind scheme.  If the local rings of $X$ are factorial, then
	$f$ is projective.
\end{proposition}	
\begin{proof}
	This proof was explained to the author by Steven Kleiman.  Fix a closed point $b_0 \in B$.
	Given any component $F \subset X_{b_0}$, I claim that we can find a line bundle $\calL$ on $X$ that has non-negative
	degree on every component of every fiber  and strictly positive degree on $F$.
	
	Pick a closed point $x \in F$ and an open affine neighborhood $U \subset X$ of that point.  By the Prime Avoidance Lemma, we can find a regular 
	function $r \in H^{0}(U, \calO_{X})$ that does not vanish on any component of $X_{b_0}$ but does vanish at $x$.  Pick a component
	$D$ of the closure of $\{ r=0 \} \subset U$ in $X$.  Then $D$ is a Cartier divisor (by the factoriality assumption) that does not 
	contain any component of any fiber $X_b$ (by construction).  Furthermore, $D$ has non-trivial intersection with $F$.  
	The associated line bundle $\calL := \calO_{X}(D)$ has the desired positivity property.
	
	Now construct one such line bundle for every irreducible component $F$ of $X_{b_0}$ and define $\calM$ to be their tensor product.  The line bundle $\calM$ is nef on 
	every fiber and ample on $X_{b_0}$.  Ampleness is an open condition, so $\calM$ is in fact ample on all but finitely many fibers of $f$.  After repeating
	the construction for each such fiber and forming the tensor product, we have constructed a $f$-relatively ample line bundle on $X$.  This completes the proof.
\end{proof} 

We now turn our attention to the moduli spaces.

\subsection{Esteves Jacobians} \label{Subsec: Esteves}
We first discuss the Esteves  moduli space of quasi-stable sheaves.  This moduli 
space  fits very naturally into the framework of the previous section.

Suppose $B$ be a locally noetherian scheme and $f \colon X \to B$ 
a projective family of curves whose fibers are geometrically reduced.
Quasi-stability is  defined in terms of a section $\sigma \colon B \to X^{\text{sm}}$ 
and a vector bundle $\calE$ on $X$ with fiber-wise integral slope 
$\deg(\calE_b)/\operatorname{rank}(\calE_b)$, which
we assume is constant as a function of $b \in B$.  Given $\sigma$ and $\calE$,
$\sigma$-quasi-stability is a numerical condition on the  multi-degree of a rank $1$, torsion-free sheaf of degree
\begin{displaymath} \label{Eqn: EstevesDegree}
	d(\calE) = d :=  -\chi(\calO_{X_b}) -\deg(\calE_{b})/\operatorname{rank}(\calE_{b}).
\end{displaymath}
For the definitions (which we will not use), we direct the reader to \cite[p.~3051]{esteves} (for a single sheaf)
and \cite[p.~3054]{esteves} (for a family).  The basic existence theorem is Theorem~A (stated on \cite[p.~3047]{esteves};
 proved in \cite[Sect.~4]{esteves}).  It states that if  $\Sheaf_{\calE}^{\sigma} \colon 
\text{$B$--$\Sch$} \to \Sets$ is the functor defined by setting $\Sheaf_{\calE}^{\sigma}(T)$ equal to the
set of isomorphism classes of $\calO_{T}$-flat, finitely presented $\calO_{X_T}$-modules
that are fiber-wise $\sigma$-quasi-stable, then there is a $B$-proper algebraic space
$\bar{J}_{\calE}^{\sigma} \to B$ of finite type that represents the \'{e}tale sheaf associated to
$\Sheaf_{\calE}^{\sigma}$. 

Strictly speaking, our definition differs from the one given in \cite{esteves} in two ways.  First,
in \cite{esteves} the author does not work with isomorphism classes of sheaves but rather with equivalence classes
under the relation given by identifying  two sheaves $\calI_1$ and $\calI_2$ on $X_{T}$ when $\calI_1$  is isomorphic
to $\calI_2 \otimes f^{*}(\calL)$ for some line bundle $\calL$ on $T$.  Zariski locally on $T$
the sheaves $\calI_1$ and $\calI_2$ are isomorphic, and it follows that the \'{e}tale
sheaf associated to $\Sheaf_{\calE}^{\sigma}$ is canonically isomorphic to the
sheaf considered in \cite{esteves}.  Second, Esteves only defines his functor 
as a functor from locally noetherian schemes to sets.  However, the functor 
$\Sheaf_{\calE}^{\sigma}$ and its associated \'{e}tale sheaf are easily seen to be
locally finitely presented.  It follows that $\bar{J}_{\calE}^{\sigma}$ represents the
\'{e}tale sheaf associated to $\Sheaf_{\calE}^{\sigma}$, rather than just the restriction of 
this sheaf  to locally noetherian schemes.

If $f$ satisfies stronger conditions, then the space $\bar{J}_{\calE}^{\sigma}$ is actually a scheme.
This is the content of Theorem~B (stated on \cite[p.~3048]{esteves}; proved on \cite[p.~3086]{esteves}).
The theorem states that if there exist sections $\sigma_1, \dots, \sigma_n \colon B \to X^{\text{sm}}$
of $f$ with the property that every irreducible component
of a fiber $X_b$ is geometrically integral and contains one of the points $\sigma_1(b), \dots, \sigma_n(b)$,
then $\bar{J}^{\sigma}_{\calE}$ is a scheme.

In the special case where $B = S$ is a strict henselian discrete valuation ring, the hypotheses
of Theorem~B are automatically satisfied.  Indeed, the locus of $k(0)$-valued
points is dense in the smooth locus $X_0^{\text{sm}}$ (Fact~\ref{Fact: SepPoints}),
which in turn is dense in $X_0$ as $X_0$ is geometrically reduced.    We may
conclude that  the irreducible components of $X_0$ are geometrically integral.
Finally, every $k(0)$-valued point of $X_0$ extends to a section 
$\sigma \colon S \to X$ (Fact~\ref{Fact: SectionsExist}), so the hypotheses
to Theorem~B are certainly satisfied.
 
We call $\bar{J}_{\calE}^{\sigma}$ the \textbf{Esteves compactified Jacobian}.  Inside of
the Esteves compactified Jacobian, we can consider the open subscheme parameterizing line bundles.
This scheme is called the \textbf{Esteves Jacobian} and denoted by $J_{\calE}^{\sigma}$.  While the scheme 
$\bar{J}_{\calE}^{\sigma}$ parameterizes sheaves, it is not naturally
a subfunctor of $\Sh$ because it does not parameterize degree $0$ sheaves.  We can, 
however, define a natural transformation $\bar{J}_{\calE}^{\sigma} \to \Sh$ by the rule
\begin{displaymath}
	\calI \mapsto \calI(-d \cdot  \sigma)
\end{displaymath}
Both Proposition~\ref{Prop: OpenImmersion} and  Theorem~\ref{Thm: Main Theorem} apply to 
 $\bar{J}_{\calE}^{\sigma}$.
\begin{corollary} \label{Cor: EstevesJacDedekind}
	Fix a Dedekind scheme $B$.  Let $f \colon X \to B$ be a projective family of geometrically reduced curves.  Let 
	$\sigma \colon B \to X^{\text{sm}}$ be a section and $\calE$  a vector bundle on $X$ with fiber-wise integral slope.
	
	Then the natural map  $J_{\calE}^{\sigma} \to \Quot$ is an open immersion.  Assume further:
	\begin{itemize}
		\item $f$ satisfies Hypothesis~\ref{Hyp: Factorial}.
	\end{itemize}
	Then  $J_{\calE}^{\sigma} = \QuotT$, and  this scheme is the N\'{e}ron model.
\end{corollary}
\begin{proof}
	By localizing, we can assume that $B=S$ is a strict henselian discrete valuation ring, in which case we
	are reduced to proving that the hypotheses of Proposition~\ref{Prop: OpenImmersion} and  Theorem~\ref{Thm: Main Theorem} 
	hold.  The scheme $J_{\calE}^{\sigma}$ is easily seen to be formally $S$-smooth.  
	 Indeed, $\sigma$-quasi-stability is a condition on fibers, so the formal smoothness
	 of $\PicZero$ implies the formal smoothness of $J_{\calE}^{\sigma}$.  The remaining hypotheses
	 of Proposition~\ref{Prop: OpenImmersion} are explicitly assumed, so we can deduce the first
	 part of the theorem.  To complete the proof, it is enough to show that $J_{\calE}^{\sigma}$
	 is stabilized by $\QuotO$.  But the action of $\QuotO$ on $J_{\calE}^{\sigma}$ is given by 
	 the tensor product against a multi-degree $0$ line bundle, so this action preserves 
	 multi-degree and hence $\sigma$-quasi-stability.
\end{proof}

Corollary~\ref{Cor: EstevesJacDedekind} implies that $J^{\sigma}_{\calE}$ is a scheme with (unique) $B$-group scheme structure that extends the group scheme structure of the generic fiber.   It is not immediate from the definition that $J^{\sigma}_{\calE}$
admits such structure, and Example~\ref{Example} shows that the group structure is special to the case of families over a
$1$-dimensional base.  The result also implies uniqueness results for the Esteves Jacobian; if $\sigma' \colon B \to X^{\text{sm}}$ 
is a second section and $\calE'$ a second vector bundle on $X$, then $J^{\sigma'}_{\calE'}$ is canonically isomorphic to $J^{\sigma}_{\calE}$.  In the next section, we will define the Simpson stable Jacobian $J_{\calL}^{0}(X)$, and this scheme is also isomorphic to $J^{\sigma}_{\calE}$ provided every slope semi-stable sheaf is stable.  It would be interesting to know if these isomorphisms extend to the compactifications. Important results along these lines can be found in \cite{viviani10} and \cite{esteves09}, but  many basic question remain unanswered.  Currently, there is not a single example of a curve $X_0 \to \Spec(k)$ such that two Esteves compactified Jacobians associated to $X_0$ are non-isomorphic.

\subsection{Simpson Jacobians} \label{Subsec: Simpson}
The hypotheses to Proposition~\ref{Prop: OpenImmersion} and 
Theorem~\ref{Thm: Main Theorem} are satisfied by certain moduli spaces of stable sheaves, which we call Simpson Jacobians.
Here we recall Simpson's construction, along with later work of Langer and Maruyama, and then apply results from Section~\ref{Sec: MainThm}.   
We restrict our attention to families of reduced curves (but see  Remark~\ref{Remark: OpenClosedFail}, and 
the discussion preceding Example~\ref{Example}).

We work over a scheme $B$ that is finitely generated over a universally Japanese ring $R$ (e.g.
$R = \mathbb{C}, \mathbb{F}_p, \mathbb{Z},...$).  Let $f \colon X \to B$  a family of curves
with  $f$-relatively  ample line bundle $\calL$, and assume the
  Euler--Poincar\'{e} characteristics $\chi(\calO_{X_b})$ and $\chi(\calL_b)$ 
are constant as functions of the base $B$.  Set $P_{d}$ equal to the polynomial
\begin{equation} \label{Eqn: HilbPoly}
	P_{d}(t) := \deg(\calL_b) \cdot t +d+ \chi,
\end{equation}
where  $\chi$ is the Euler--Poincar\'{e} characteristic of a fiber of $f$ and $\deg(\calL_b)$ is the degree
of the restriction of $\calL$ to a fiber.  This is the Hilbert polynomial of a degree $d$ line bundle.

Given this data, Simpson constructed an associated moduli space provided $R=\mathbb{C}$.  The Simpson moduli space $\operatorname{M}(\calO_{X}, P_d)$ parameterizes slope semi-stable sheaves with Hilbert
polynomial $P_d$.  (See \cite[p.~54--56]{simpson94} for the definition of semi-stability).  To be precise, define 
$\operatorname{M}^{\sharp}(\calO_{X}, P_d)$ to be the functor whose $T$-valued points are isomorphism classes
of $\calO_{T}$-flat, finitely presented $\calO_{X_{T}}$-modules whose fibers are $\calL$-slope semi-stable sheaves
with Hilbert polynomial $P_d$.  The main existence result (\cite[Thm.~1.21]{simpson94}) asserts that there is a projective scheme 
$\operatorname{M}(\calO_{X}, P_{d})$ that corepresents $\operatorname{M}^{\sharp}(\calO_{X}, P_d)$.  Inside of 
$\operatorname{M}(\calO_{X}, P_{d})$, we may consider the open subscheme 
$\operatorname{M}^{\text{st}}(\calO_{X}, P_{d})$ parameterizing $\calL$-slope stable sheaves.  The stable locus is a 
fine moduli space: its $\mathbb{C}$-valued points are in natural bijection with the isomorphism classes of $\calL$-slope stable
sheaves with Hilbert polynomial $P_{d}$, and  \'{e}tale locally on $\operatorname{M}^{\text{st}}(\calO_{X}, P_{d})$,
the product $X \times_{B} \operatorname{M}^{\text{st}}(\calO_{X}, P_{d})$ admits a universal family of sheaves.  The reader may check that
these conditions are equivalent to the condition that $\operatorname{M}^{\text{st}}(\calO_{X}, P_{d})$ represent the \'{e}tale sheaf associated the 
functor parameterizing stable sheaves.  While Simpson only considers the case $R = \mathbb{C}$, later 
work of Langer (\cite[Thm.~4.1]{langer04a}, \cite[Thm.~0.2]{langer04b})
and Maruyama (\cite{maruyama}) extends these results to the case where $R$ is an arbitrary universally Japanese ring.

Let us now specialize to the case where $B$ is a Dedekind scheme.  When $f$ has reducible fibers, $\operatorname{M}^{\text{st}}(\calO_{X}, P_{d})$ may contain points corresponding to sheaves that are not rank $1$ (see \cite[Ex.~2.2]{lopez05}).  This is potentially a major source of confusion: the term ``rank" is used in a different way in \cite{simpson94},
and the sheaves parameterized by $\operatorname{M}^{\text{st}}(\calO_{X}, P_{d})$ are rank $1$ in Simpson's sense but not necessary in the sense
used here.  

We avoid these sheaves.  Define the \textbf{Simpson stable Jacobian}  $J_{\calL}^{d}$ of degree $d$ to be the locus of 
stable line bundles in $\operatorname{M}^{\text{st}}(\calO_{X}, P_{d})$ (which is an open subscheme by \cite[Lemma~5.12(a)]{altman80}).  We define the \textbf{Simpson stable compactified Jacobian} $\bar{J}_{\calL}^{d}$ to be the subset of  
the stable locus $\operatorname{M}^{\text{st}}(\calO_{X}, P_{d})$ that corresponds to  pure, rank $1$ sheaves. (Warning: the compactified Jacobian is a $B$-proper scheme when 
every semi-stable pure sheaf with Hilbert polynomial $P_{d}$ is stable but not in general!)  

When the fibers of $X \to B$ are \emph{geometrically reduced}, a minor modification of the proof of \cite[Lemma~8.1.1]{pand} shows that the subset $\bar{J}_{\calL}^{d} \subset \operatorname{M}^{\text{st}}(\calO_{X}, P_{d})$ is closed and open, and  hence  has a natural scheme structure.  We record this.

\begin{lemma} \label{Lemma: OpenAndClosed}
	Assume the fibers of $f \colon X \to B$ are geometrically reduced.  Then the subset $\bar{J}_{\calL}^{d}$ is closed and open in $\operatorname{M}^{\text{st}}(\calO_{X}, P_{d})$.
\end{lemma}
\begin{proof}
	The main point to prove is that a $1$-parameter family of line bundles cannot specialize 
	to a pure sheaf that fails to have rank $1$, and this is shown by examining the leading
	term of the Hilbert polynomial. To begin, we may cover
	$\operatorname{M}^{\text{st}}(\calO_{X}, P_{d})$ by \'{e}tale morphisms 
	$\operatorname{M} \to \operatorname{M}^{\text{st}}(\calO_{X}, P_{d})$ 
	with the property that a universal family $\calI_{\text{uni.}}$ on $\operatorname{M} \times_{B} X$
	exists. It is enough to 
	verify the claim after passing from $\operatorname{M}^{\text{st}}(\calO_{X}, P_{d})$
	to an arbitrary such  scheme,
	and so for the remainder of the proof we work with $\operatorname{M}$ in place of 
	$\operatorname{M}^{\text{st}}(\calO_{X}, P_{d})$. We will also abuse notation by denoting
	the pullback of 	$\bar{J}_{\calL}^{d}$ under $\operatorname{M} \to \operatorname{M}^{\text{st}}(\calO_{X}, P_{d})$
	by the same symbol $\bar{J}_{\calL}^{d}$.

	We first need to check that $\bar{J}^{d}_{\calL} \subset \operatorname{M}$ is constructible, so that
	we can make use of the valuative criteria.  Given $m \in M$ mapping to $b \in B$, the condition that the fiber $I_m$ is 
	rank $1$ is just the condition that the restriction of $I_m$ to $X_{b}^{\text{sm}}$ is a line bundle.  Constructibility thus
	follows from \cite[9.4.7]{ega43} applied to $\operatorname{M} \times_{B} X^{\text{sm}} \to \operatorname{M}$.
	
	To complete the proof, it is enough to prove that $\bar{J}_{\calL}^{d}$ is closed under specialization and generalization.
	Thus, we pass from $\operatorname{M}$ to a discrete valuation ring $T$ mapping to $\operatorname{M}$.   If
	$\calI$ is the sheaf  on $X_{T}$ given by the pulback of the universal family, then we need to show
	that the generic fiber of $I_\eta$ is rank $1$ if and only if the special fiber $I_0$ is.
	
	Consider the leading term of the Hilbert polynomial of these sheaves.  On one hand, we assumed that
	the leading term is $\deg(\calL_0)$.  On the other hand, this term can be computed in terms of the generic
	ranks of the sheaves (using \cite[2.5.1]{altman79}).  Suppose that $x_1, \dots, x_n$ are the generic points of the special fiber and 
	$y_1, \dots, y_m$ are the generic points of the generic fiber.  We write $\ell_{x_i}(I_0)$ (resp. $\ell_{y_i}(I_{\eta})$)
	for the length of the stalk of $I_0$ at $x_i$ (resp. $I_{\eta}$ at $y_i$).
	The leading term of the common Hilbert polynomial can be expressed as:
	\begin{align} \label{Eqn: LeadingTerm}
		\deg(\calL_0) =& \deg_{x_1}(\calL_0) \cdot \ell_{x_1}(I_0) + \dots + \deg_{x_n}(\calL_0) \cdot  \ell_{x_n}(I_0) \\
				=&  \deg_{y_1}(\calL_\eta) \cdot  \ell_{y_1}(I_\eta) + \dots + \deg_{y_m}(\calL_\eta) \cdot  \ell_{y_m}(I_\eta) \label{Eqn: LeadingTermLineTwo}
	\end{align}
	The degree $\deg(\calL_0)$ is also the sum of the partial degrees $\deg_{x_i}(\calL_0)$, so Eq.~\eqref{Eqn: LeadingTerm} remains
	valid if we replace every $\ell_{x_i}(I_0)$ with $1$, and similarly for Eq.~\eqref{Eqn: LeadingTermLineTwo}.
	
	Because the fibers are reduced, the lengths $\ell_{x_i}(I_0)$ and $\ell_{y_j}(I_{\eta})$ are just the dimension of the stalk
	of $I_0$ at $x_i$ and the stalk of $I_{\eta}$ at $y_j$ respectively.  In particular, these
	numbers are upper semi-continuous.  
	
	Suppose first that $I_0$ is rank $1$.  By semi-continuity, we have $\ell_{y_i}(I_{\eta}) \le 1$.
	If some inequality was strict, say $\ell_{y_1}(I_{\eta})=0$, then
	\begin{align*}
		\deg(\calL_{0}) =& \deg_{y_1}(\calL_\eta) + \deg_{y_2}(\calL_\eta)  + \dots + \deg_{y_m}(\calL_\eta) \\
						>& \phantom{\deg_{y_1}(\calL_\eta) + } \deg_{y_2}(\calL_\eta)  + \dots + \deg_{y_m}(\calL_\eta) \\
					\ge& \deg_{y_1}(\calL_\eta) \cdot  \ell_{y_1}(I_\eta) + \dots + \deg_{y_m}(\calL_\eta) \cdot  \ell_{y_m}(I_\eta) \\
					=& \deg(\calL_{0}).
	\end{align*}
	This is absurd!  Thus, we must have $\ell_{y_i}(I_{\eta}) = 1$ for all $y_i$ and $I_{\eta}$ is rank $1$.  Similar reasoning shows
	that if $I_{\eta}$ is rank $1$, then $I_0$ is rank $1$.
\end{proof}

\begin{remark} \label{Remark: OpenClosedFail}
	The hypothesis that the fibers of $f$ are geometrically reduced is necessary.  Indeed, the moduli space $\operatorname{M}^{\text{st}}(\calO_{X}, P_d)$
	was described in \cite{dawei} in the case that $X$ is a non-reduced curve whose reduced subscheme $X_{\text{red}}$ is smooth and whose nilradical $\calN$ is square-zero (i.e. $X$ is a ribbon).
	Using that description it is easy to produce examples where $\bar{J}_{\calL}^{d} \subset \operatorname{M}^{\text{st}}(\calO_{X}, P_d)$ is not closed (e.g. take $d$ equal to $0$, $X$ to have
	even genus, and $X_{\text{red}}$ to have genus $1$).  The points of the complement in the closure correspond to stable vector bundles on $X_{\text{red}}$.
\end{remark}

We now apply Proposition~\ref{Prop: OpenImmersion} and Theorem~\ref{Thm: Main Theorem} to the Simpson
Jacobians. 

\begin{corollary} \label{Cor: SimpIsNeron}
	Fix a Dedekind scheme that is finitely generated over a universally Japanese ring.  Let $f \colon X \to B$ 
	be a family of geometrically reduced curves.  Let $\calL$ be  $f$-relatively ample line bundle.
	
	Then the natural map $J^{0}_{\calL}(X) \to \QuotT$ is an open immersion.  Assume further that both of the following conditions hold:
	\begin{itemize}
		\item every $\calL$-slope semi-stable rank $1$, torsion free sheaf of degree $0$ is $\calL$-slope stable;
		\item $f$ satisfies Hypothesis~\ref{Hyp: Factorial}.
	\end{itemize}
	Then $J^{0}_{\calL}(X)=\QuotT$ and this scheme is the N\'{e}ron model.
\end{corollary}
\begin{proof}
	The local existence of a universal family (\cite[Thm.~2.1(4)]{simpson94}) implies that 
there is a  natural transformation $\bar{J}_{\calL}(X) \to \Sh$ with the property
that $J_{\calL}(X)$ is the pre-image of $\PicZero \subset \Sh$.  Furthermore, the slope stability condition
is a fiber-wise condition, so a modification of the argument given in Corollary~\ref{Cor: EstevesJacDedekind} 
completes the proof.
\end{proof}

\begin{remark}
A minor generalization of Corollary~\ref{Cor: SimpIsNeron} can be obtained by allowing 
for moduli spaces of degree $d$ lines bundles, with $d \ne 0$.   
If we are given a line bundle $\calM$ on $X$ with fiber-wise degree $d$, then 
there is an associated map $J^{d}_{\calL}(X) \to \Quot$ that extends the map on the 
generic fiber given by tensoring with $\calM^{-1}$.  With only notational changes the 
previous corollary generalizes to a statement about this map.  
\end{remark}

Corollary~\ref{Cor: SimpIsNeron} is, of course, only of interest when there exists an $\calL$ such that
$\calL$-slope stability coincides with $\calL$-slope semi-stability.  Thus, we ask: when does such a 
$\calL$ exist?  A comprehensive discussion of this question would require a  digression
on stability conditions, so we limit ourselves to reviewing known results about a single curve $X_0$
over an algebraically closed field.  When $X_0$ is integral, the stability condition is vacuous, so every 
 ample $\calL_0$ has the desired property.  If $X_0$ is reducible and has only nodes as singularities,
 then  Melo and Viviani have proven the existence of a suitable $\calL_0$ (\cite[Prop.~6.4]{viviani10}).  
Stability conditions on reduced, genus $1$ curves were analyzed by  L{\'o}pez-Mart{\'{\i}}n in \cite{lopez05}. 
She exhibits curves $X_0$ with the property that there is no $\calL_0$ such that every $\calL_0$-slope semi-stable,  
pure, rank $1$ sheaf degree $0$ is stable, but a suitable $\calL_0$ always exists if one considers sheaves of fixed 
degree $d \ne 0$.  Finally, stability conditions for a ribbon were analyzed in \cite{dawei}. On a ribbon, the stability condition is independent of $\calL_0$, and for rational ribbons, slope stability coincides with slope semi-stability precisely when the genus $g$ is even.  It would be desirable to have a general result asserting (non-)existence of a suitable $\calL_0$.

\subsection{Genus $1$ Curves} \label{Subsec: Genus1}
The N\'{e}ron model of the Jacobian of a genus $1$ curve can be quite complicated (e.g. see \cite{qing}), but 
these complications do not arise if the family admits a section. Suppose $B$ is a Dedekind scheme and
$f \colon X \to B$ is a family of curves such that the total space $X$ is regular and the generic fiber $X_{\eta}$
is smooth.  If  $\sigma \colon S \to X^{\text{sm}}$  is a section 
contained in the smooth locus, then there is a canonical identification of the smooth locus $X^{\text{sm}}$ with the N\'{e}ron 
model $\Neron$ of the Jacobian of $X_{\eta}$.  Here we examine how this fact fits into the preceding framework.

\begin{definition} \label{Def: UniFam}
	Let $f \colon X \to B$ be a family of genus $1$ curves over a Dedekind scheme and
$\sigma \colon B \to X^{\text{sm}}$ a section contained in the smooth locus.  We define 
a sheaf $\calI_{\text{uni.}}$ on  $X \times_{B} X$ by the formula
\begin{equation}
	\calI_{\text{uni.}} := \calI_{\Delta}( \pi_{1}^{*}(\sigma) + \pi_{2}^{*}(\sigma)).
\end{equation}
Here $\calI_{\Delta}$ is the ideal sheaf of the diagonal and $\pi_1, \pi_2 \colon X \times_{B} X \to X$ are the projection maps.
\end{definition}

The sheaf  $\mathcal{I}_{\text{uni.}}$ determines a transformation $X \to \Sh$ that realizes $X$ as a moduli space of sheaves over itself.   Proposition~\ref{Prop: OpenImmersion} and Theorem~\ref{Thm: Main Theorem} apply to this moduli space.

\begin{corollary} \label{Cor: Genus1Neron}
	Fix a Dedekind scheme $B$.  Let $f \colon X \to B$ be a family of genus $1$ curves.  Let $\sigma \colon B \to X^{\text{sm}}$ be a 
	section.  
	
	Then the natural map $X^{\text{sm}} \to \Quot$ is an open immersion.  Assume further:
	\begin{itemize}
		\item $f$ satisfies Hypothesis~\ref{Hyp: Factorial}.
	\end{itemize}
	Then $X^{\text{sm}} = \QuotT$ and this scheme is the N\'{e}ron model.
\end{corollary}

Let us consider the special case where $B$ is a discrete valuation ring, $X$ is a minimal regular surface,
and the residue field $k(0)$ is algebraically closed.  The possibilities for the special fiber $X_0$ are given by the 
Kodaria--N\'{e}ron classification (\cite{kodaira} and \cite{neron64}; see \cite[p.~353-354]{silverman} 
for a recent exposition).  The reduced curves appearing in the classification are the Reduction
Types $\operatorname{{\uppercase\expandafter{\romannumeral 01}}}_{n}$,
$\operatorname{{\uppercase\expandafter{\romannumeral 02}}}$, $\operatorname{{\uppercase\expandafter{\romannumeral 03}}}$,
$\operatorname{{\uppercase\expandafter{\romannumeral 04}}}$.  In these cases, one may show that
the induced morphism $X \to \Sh$ identifies $X$ with the Esteves compactified Jacobian
$\bar{J}_{\calO}^{\sigma}$.  

In every  remaining case (Reduction Type $\operatorname{{\uppercase\expandafter{\romannumeral 01}}}^{*}_{n}$,
$\operatorname{{\uppercase\expandafter{\romannumeral 02}}}^{*}$, 
$\operatorname{{\uppercase\expandafter{\romannumeral 03}}}^{*}$, or
$\operatorname{{\uppercase\expandafter{\romannumeral 04}}}^{*}$) the morphism $X \to \Sh$ is 
not a special case of the fine moduli spaces discussed in the previous two sections.  Indeed, the special fiber
$X_0$ is non-reduced, so the Esteves Jacobian of $X$ is not defined.  In Section~\ref{Subsec: Simpson},
we reviewed Simpson's moduli space $\operatorname{M}^{\text{st}}(\calO_{X}, P_{d} )$ of stable sheaves, but the image of $X \to \Sh$ cannot be described as a closed subscheme of that space.  The reason is that slope stable sheaves are simple, but some fibers of $\calI_{\text{uni.}}$ are not simple.  Specifically, if $p_0 \in X_0$ lies on the intersection of 
two components, then the fiber of $\calI_{\text{uni}}$ of $p_0$ fails to be simple.  This can be seen as follows.
This fiber is the sheaf  $\calI_{p_0}( -\sigma(0) )$, where $\calI_{p_0}$ is the  ideal of $p_0$.  If $\nu \colon X'_0 \to X_0$ is 
 the blow-up of $X_0$ at $p_0$, then one may show that
$H^{0}(X'_0, \calO_{X_0'})$ is canonically isomorphic to the endomorphism ring of $\calI_{p_0}( -\sigma(0) )$.  An inspection
of the Kodaria--N\'{e}ron table shows that $X'_0$ is disconnected, so $H^{0}(X'_0, \calO_{X'_0})$ does not equal $k(0)$
and $\calI_{p_0}( -\sigma(0) )$ is not simple.

Corollary~\ref{Cor: Genus1Neron} provides a partial answer to a question posed in the introduction:
What are the maximal subfunctors $J$ of $\PicZero$ represented by a separated $B$-scheme?
When $X$ is, say, regular, a strong result  one could hope for is that
there is always a subfunctor $\bar{J}$ of $\Sh$ satisfying the 
hypotheses of Theorem~\ref{Thm: Main Theorem}.  The line bundle locus $J \subset \bar{J}$ in such a functor 
has the property that $J \to \QuotT$ is an isomorphism, and hence $J$ is maximal.  Corollary~\ref{Cor: Genus1Neron} 
shows that such a $\bar{J}$ exists when $f \colon X \to S$ is a family of genus $1$ curves that
$f$ admits a section.  Similarly, the Esteves compactified Jacobian represents a suitable subfunctor when 
$f$ has geometrically reduced fibers and admits a section. In general, however, the hope is too optimistic: 
Raynaud's family, mentioned at the end of  Section~\ref{Sec: MainThm}, has that property that no such $\bar{J}$ can exist.  

The question of describing maximal subfunctors $J$ is most interesting when $f$ has non-reduced fibers.  The slope stable line bundles form a subfunctor $J \subset \PicZero$ represented by a $S$-separated scheme, but our discussion of genus $1$ families together with Remark~\ref{Remark: OpenClosedFail} suggest that we should consider other methods for constructing a suitable $J$ when $f$ has non-reduced fibers.

In a different direction, one nice property of the moduli spaces described by Corollary~\ref{Cor: Genus1Neron} 
is that their geometry is very simple.  We use these spaces to provide an 
example showing that a family $J \to B$ of Esteves Jacobians over a regular $2$-dimensional 
base may not have group scheme structure.  

\begin{example} [N\'{e}ron models in $2$-dimensional families] \label{Example}
	We will construct a $2$-dimensional family $f \colon X \to B$ of plane cubics and an associated Esteves
	Jacobian $J \to B$ with the property that the group law on the locus $J_{U} \to U$ parameterizing non-singular
	cubics does not extend over all of $B$.  Furthermore, the family is constructed in such a 
	way that a dense open subset of $B$ is covered by non-singular curves $C$ with the property that the restriction 
	$X_{C}$ of $X$ to $C$ is regular, so $J_{C}$ is the N\'{e}ron model of its generic fiber (and in particular
	admits group scheme structure that extends the group scheme structure over $C \cap U$).  Thus, the 
	N\'{e}ron models fit into a $2$-dimensional family, but their group scheme structure does not.
	
	The idea is as follows.  The family we construct has a reducible element $X_{b_0} \to b_0$ with
	the property that, for every non-singular curve $C \subset B$ passing through $b_0$ such
	 that $X_{C}$ is regular, the restriction of the Esteves
	Jacobian $J_{C}$ is the N\'{e}ron model of its generic fiber.  The 
	fiber $J_{b_0}$ inherits a group law from this N\'{e}ron model, and we show 
	explicitly that this group law depends on the particular choice of $C$.  But,
	if the group law on $J_{U}$ extended to $J$, then all the different group 
	laws on $J_{b_0}$ coming from the different curves $C$ would be the
	restriction of one common group law on $J$, which is absurd.  We now construct the
	family.
	
	We work over an algebraically closed field $k$.  The family $X \to B$ will be a net of plane cubics.
	Let $X_0 \subset \mathbb{P}^{2}_{k}$ be a reducible plane cubics that is the union of
	a smooth quadric $Q_0$ and a line $L_0$ that meet in two distinct points.  (See Fig.~\ref{Figure: Pencil}.)
	Fix two general points $p_1, p_2 \in L_0(k)$ on the line and one general point $q_1 \in Q_0(k)$ on the quadric.  
	Say that $F \in H^{0}( \mathbb{P}^{2}_{k}, \calO(3))$ 
	is an equation for $X_0$ and $G, H \in H^{0}( \mathbb{P}^{2}_{k}, \calO(3))$ are two
	general cubic equations that vanish on all of the points $p_1,p_2, q_1$.  We will 
	work with the net $V := \langle F, G, H \rangle \subset H^{0}(\mathbb{P}^{2}_{k}, \calO(3))$
	and the associated family of curves
	\begin{align}
		X	:=	& \{ (p, [r, s, t]) \colon r \cdot F(p) + s \cdot G(p) + t \cdot F(p) = 0 \} \\
			\subset	& \mathbb{P}^{2}_{k} \times \mathbb{P}^{2}_{k}. \notag
	\end{align}
	There are two obvious morphisms $e, f \colon X \to \mathbb{P}^{2}_{k}$ given by the
	two projections.  If we set $B := \mathbb{P}^{2}_{k}$ equal to the plane, then the
	second morphism $f \colon X \to B$ realizes $X$ as a family 
	of genus $1$ curves with $X_0 = f^{-1}(b_0)$, $b_0 := [1,0,0]$.  Corresponding to the
	points $p_1, p_2, q_1 \in X_0(k)$ are three section $\sigma_1, \sigma_2, \tau_1 \colon 
	B \to X^{\text{sm}}$, which lie in the smooth locus by the generality assumption.
	
	Another application of the generality assumption shows that the fibers of $f$ are reduced,
	so we can form the Esteves Jacobian $J := J^{\sigma_1}_{\calE}$, where $\calE = \calO_{X}$.
	If $\calL_0$ is a line bundle on $X_0$, then the quasi-stability condition is that the bidegree $(\deg(\calL_{L_0}), \deg(\calL_{Q_0}))$ 
	equals $(0,0)$ or $(1, -1)$.  Now we assume $J \to B$
	is a group scheme and derive a contradiction.

	Suppose that we are given a general line $C \subset B$ in the plane that contains
	$b_0$.  Such a line corresponds to a $2$-dimensional linear subspace of the form $W := \langle F, G_{C} \rangle \subset V$
	for some $G_{C} \in V$.  Invoking generality again, the base locus
	\begin{equation} \label{Eqn: BaseLocus}
		\{ p \in \mathbb{P}^{2}_{k} \colon F(p) = G_{C}(p) = 0 \}
	\end{equation}
	consists of $9$ distinct points.  The first projection map $e \colon X \to \mathbb{P}^{2}_{k}$ realizes $X_{C}$ as
	the blow-up of the plane at these points, so $X_{C}$ is regular, and thus $J_{C}$ is the N\'{e}ron model
	of its generic fiber.  We now study the group of sections of $J_{C} \to C$.
	
	The base locus \eqref{Eqn: BaseLocus} includes the points $p_1, p_2, q_1$.  In addition to the points
	$p_1, p_2$, a unique third point of the base locus must lie on the line $L_0$.  Let us label that point
	$p_{C}$ and write $\sigma_{C} \colon C \to X^{\text{sm}}$ for the corresponding section.  
	
		 \begin{figure}[ht]
 \centering
	\includegraphics[width=0.40\textwidth]{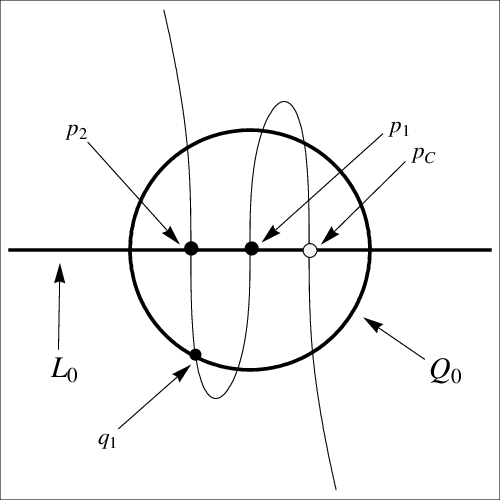}
 \caption{The pencil $X_{C}$.}
 \label{Figure: Pencil}
\end{figure}

	 %Indeed, an inspection shows that the multi-degree of each line bundle
%	on $X_{b_0}$ is $(1, -1)$, which we have observed is a quasi-stable multi-degree, and the quasi-stability condition on an 
%	irreducible fiber $X_{b} := f^{-1}(b)$ is vacuous.  The generality of $C$ implies that $b_0$ is the only point
%	of $C$ with reducible fiber, so this establishes the claim.

	Now consider the following  line bundles on $X_{C}$:
	\begin{gather*} 
		\calL_1 := \calO_{X}(\sigma_1 - \tau_1), \\
		\calL_2 := \calO_{X}(\sigma_2 - \tau_1), \\
		\calL_{C} := \calO_{X}(\sigma_{C} - \tau_1), \\
		\calM := \calO_{X}(1) \otimes \calO_{X}(-3 \cdot \tau_1).
	\end{gather*}
	These lines bundles are all $\sigma_1$-quasi-stable. If we let  $g_1, g_2, g_C, h \in J_{C}(C)$  
	respectively correspond to $\calL_1, \calL_2, \calL_{C}, \calM$, then I claim we have
	\begin{equation} \label{Eqn: Summing}
		g_1 + g_2 + g_{C} = h.
	\end{equation}
	Indeed, it is enough to verify the claim after passing to the generic fiber of $J \to C$, where the equation
	is just the statement that the points $p_1, p_2, p_{C}$ all lie on a line (the line $L_0$).  Now suppose that
	$J \to \mathbb{P}^2_{k}$ admits a group law extending the group law of the generic fiber.  Then the 
	specialization of  Equation~\eqref{Eqn: Summing} to $J_{b_0}$ holds for all $C$ simultaneously.  In 
	particular, the isomorphism class of the line bundle  
	$\calO_{X_{b_0}}( p_{C} - q_1)$ is independent of the particular line $C \subset \mathbb{P}^{2}_{k}$ chosen.  
	But this is absurd: for distinct
	general lines $C_1, C_2$, the points $p_{C_1}$ and $p_{C_2}$ (and hence the associated line bundles) are distinct!  This completes
	our discussion of this example.
\end{example}

This example is particularly interesting in light of \cite{oda}.  The authors of that paper consider the case of a family of  nodal curves $f \colon X \to B$ over a suitable Dedekind scheme with the property that $X$ is regular.  Let $J_{\eta}$ be the Jacobian of the generic fiber.  Given a closed point $0 \in B$, they prove that the special fiber $\Neron_0$ of the N\'{e}ron model of $J_{\eta}$ depends only on the curve $X_0$ and not the particular family $f$ (\cite[Cor.~14.4]{oda}).  This result must be interpreted with care: in our example, the group law depends  on a particular choice of family, but any two such group laws define isomorphic group schemes.

\section{Acknowledgments}
The results of this paper are a part of the author's thesis. He would like to thank his advisor Joe Harris  for his invaluable help with this work.  This thesis was carefully read by Filippo Viviani and Margarida Melo.  Their feedback was very helpful to the author in writing this paper, and he would like to thank them.  We would like to thank Steven Kleiman for explaining the proof of Proposition~\ref{Prop: Projective}.  We would like to thank Valery Alexeev for a useful discussion of stability conditions.  We would also like to thank the anonymous referee, Bryden Cais, Lucia Caporaso, Eduardo Esteves, Robert Lazarsfeld, and Dino Lorenzini for feedback about early drafts of this paper.   Finally, we thank  Bhargav Bhatt, Matthew Satriano, and Karen Smith for conversations that were helpful in clarifying  technical aspects of this paper.

%Lang's question about good completions provided motivation for this work.  During the early stages, Siegfried Bosch, Ching-Li Chai, Brian Conrad, Gregory Call, Marc Hindry, Klaus K{\"u}nnemann, Joseph Silverman explained the current status of this question to the author, and he would like to thank them for doing so. 
 
\bibliography{Neron}

\end{document}